\documentclass{cimart}

\usepackage{caption}
\usepackage{subcaption}

\newcommand{\ba}{\begin{array}}
\newcommand{\bb}{\begin{equation}}
\newcommand{\ea}{\end{array}}
\newcommand{\ee}{\end{equation}}

\title{A look on equations describing pseudospherical surfaces}

\authors{Igor Leite Freire}

\authorinfo{Federal University of São Carlos, Brazil}{igor.freire@ufscar.br}

\abstract{%
    We revisit the notion of equations describing pseudospherical surfaces, starting from the works by Sasaki, whose roots were influenced by the AKNS system, the works by Chern and Tenenblat, until current research topics in the field relating to Cauchy problems involving these equations and their geometric consequences.
    }



\VOLUME{34}
\YEAR{2026}
\ISSUE{2}
\NUMBER{6}
\DOI{https://doi.org/10.46298/cm.15963}
\editinfo{July 1, 2025}{January 15, 2026}{Tiago Macedo, Jaqueline Mesquita, Mariel Saez, and Rafael Potrie}

\acknowledgments{I would like to thank Professor Tiago Macedo for his invitation to contribute to CM with this review paper. I am also thankful to CNPq (grant n. 310074/2021-5) and FAPESP (grant n. 2024/01437-8) for financial support. I am also thankful to the referees for their comments and suggestions that have improved an earlier version of this manuscript. This article is dedicated to Professor Keti Tenenblat on the occasion of her 80th birthday.}

\begin{document}

\section{Introduction}

The deep relations between partial differential equations (PDEs) and geometry date back to the XIX century, when a non-linear partial differential equation was shown to be related to the Gaussian curvature of a surface. Since then, a number of non-linear equations have been proved to be of crucial relevance to better understand geometric aspects of surfaces.  

While in the XIX century the connections between PDEs and geometry was established, in the course of the XX century we witnessed a deep relation between equations coming from mathematical physics, notably from hydrodynamics, and their connections with the geometry of surfaces with constant Gaussian curvature. 

After the WW II, the KdV equation was rediscovered (for some historical information, see the books by Drazin and Johnson \cite{drazin} or Kasman \cite{kasman}), and during the 1960's it received considerable attention due to the discovery of the soliton by Zabusky and Kruskal \cite{zab}, and the techniques proposed by Gardner {\it et. al.} for solving it \cite{gardprl}. These are two of the most prominent\footnote{In order to infer the relevance of \cite{zab} and \cite{gardprl}, according to Scopus, they have received 3,340 and 3,706 citations, respectively. Data checked on December 12, 2025.} works in the field of integrable systems developed in the 60s. After these works the KdV equation attracted spotlights and its relevance was reinforced by a series of works published in the Journal of Mathematical Physics by Gardner, Miura, Kruskal and co-workers, see \cite{miura} and references therein. See also \cite{drazin,kasman}.

The works \cite{gardprl,miura,zab} were rather influential in the field of non-linear science during the late 60s and early 70s, leading to the development of a new method for finding equations sharing the nice properties of the KdV equation: in \cite{akns} the authors considered the following linear system
\bb\label{1.0.1}
\partial_x\begin{pmatrix}
v_1\\
\\
v_2
\end{pmatrix}=
\begin{pmatrix}
-i\zeta& q\\
\\
r& i\zeta
\end{pmatrix}
\begin{pmatrix}
v_1\\
\\
v_2
\end{pmatrix},
\ee
\bb\label{1.0.2}
\partial_t\begin{pmatrix}
v_1\\
\\
v_2
\end{pmatrix}=
\begin{pmatrix}
A& B\\
\\
C& D
\end{pmatrix}
\begin{pmatrix}
v_1\\
\\
v_2
\end{pmatrix},
\ee
where, $q$ and $r$ are functions of $(x,t)$ whereas $A$, $B$, $C$ and $D$ depend on $x,\,t$ and $\zeta$, which may eventually depend on $t$. The field variables $x$ and $t$ can be regarded as space and time, respectively.

Equation \eqref{1.0.1} is a scattering problem while \eqref{1.0.2} is an evolution problem. From the cross differentiation and the integrability condition
$$
\partial_t\partial_x\begin{pmatrix}
v_1\\
\\
v_2
\end{pmatrix}=\partial_x\partial_t\begin{pmatrix}
v_1\\
\\
v_2
\end{pmatrix},
$$
\eqref{1.0.1} and \eqref{1.0.2} give
$$
\begin{pmatrix}
\partial_x A& \partial_x B\\[2pt]  
\partial_x C& \partial_x D
\end{pmatrix}-
\begin{pmatrix}
0& \partial_t q\\[2pt]
\partial_t r& 0
\end{pmatrix}
= i\partial_t\zeta
\begin{pmatrix}
-1& 0\\[2pt]
0& 1
\end{pmatrix}
+
\begin{pmatrix}
qC-rB& q(D-A)-2i\zeta B\\[2pt]
r(A-D)+2i\zeta C& rB-qC
\end{pmatrix}.
$$

The $2\times2$ matrix in \eqref{1.0.1} is traceless, as are all but the first in the above equation. If we impose that the second matrix in \eqref{1.0.2} be traceless, then all matrices we have so far will share this property. Without loss of generality, we may assume $D=-A$.

In addition, if we assume that the parameter $\zeta$ is isospectral, that is, it is independent of time, then we get the system of PDEs
\bb\label{1.0.3}
\left\{\ba{lcl}
\partial_x A&=&qC-rB,\\
\\
\partial_xB+2i\zeta B&=&\partial_tq-2Aq,\\
\\
\partial_xC-2i\zeta C&=&\partial_tr+2Ar.
\ea
\right.
\ee

The reader may be wondering: {\it What is the relevance of \eqref{1.0.3}?} This is a very good question.

In \cite{akns} it was shown that a number of very important models, such as the KdV and the sine-Gordon equations, can be obtained from \eqref{1.0.3} depending on the choices of the involved functions. Moreover, equations arising from the compatibility condition \eqref{1.0.3} are said to be integrable ({\it à la AKNS}).

Let us give an example illustrating how an equation can be obtained from \eqref{1.0.3}. 

\begin{example}
    Consider the choice
    $$
    A=\frac{i\cos{u}}{4},\,\,r=-q=\frac{u_x}{2},\,\,B=C=\frac{i\sin{u}}{4\zeta}.
    $$

    Then \eqref{1.0.3} is equivalent to the sine-Gordon (sG) equation
    \bb\label{1.0.4}
    u_{tx}=\sin{u}.
    \ee
\end{example}

It is worth mentioning that the sG equation had been discovered more than a century before \cite{akns} brought it to light, in a purely geometric context. In fact, the sG plays a vital role in the context of the geometry of surfaces with constant and negative Gaussian curvature \cite[Chapter 1]{rogers} and from its solutions it is possible to construct surfaces like those shown in Figure \ref{fig1}.

\begin{figure}[ht]
	\centering
	\begin{subfigure}{0.4\linewidth}
		\includegraphics[width=0.65\linewidth]{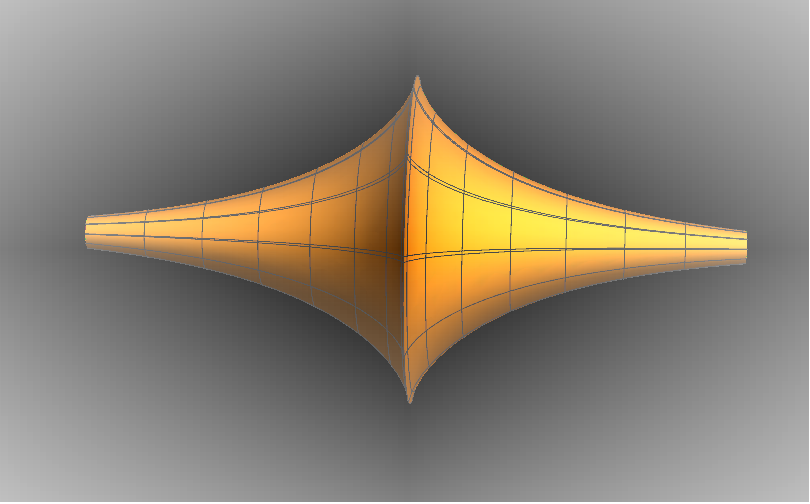}
		\label{subfig 6.3.1.a}
	\end{subfigure}
     \begin{subfigure}{0.45\linewidth}
        \includegraphics[width=0.65\linewidth]{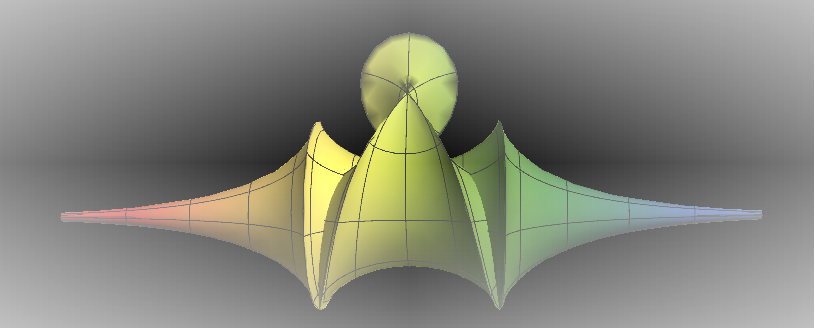}
		\label{subfig 6.3.1 c}
	\end{subfigure} 
	\begin{subfigure}{0.45\linewidth}
		\includegraphics[width=.65\linewidth]{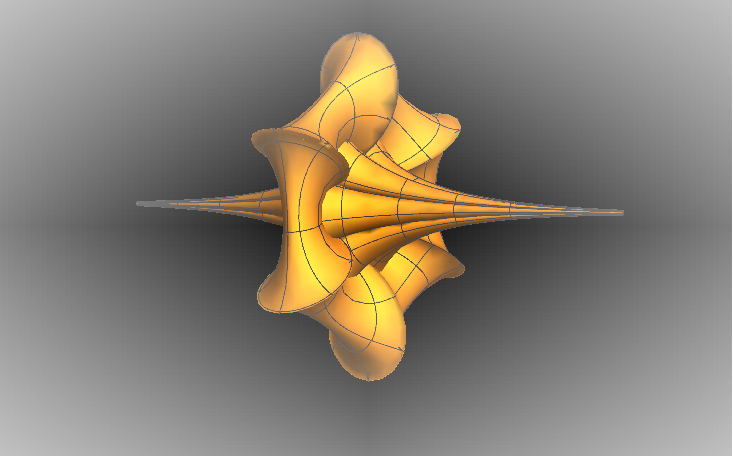}
		\label{subfig 6.3.1 b}
	\end{subfigure}
    \caption{Surfaces of Gaussian curvature ${\cal K}=-1$. In \cite[Chapter 1]{rogers} it is shown how the above surfaces can be obtained from the solutions of the sine-Gordon equation \eqref{1.0.4}.}\label{fig1}
\end{figure}

Let us now introduce a geometric view on \eqref{1.0.1}--\eqref{1.0.2}. By $\{dx,dt\}$ we denote the usual dual basis of $\mathbb{R}^2$. Consider eight one-forms $\Omega_{ij},\,\Psi_{ij}$, $1\leq i,j\leq 2$ and define the $2\times 2$ one-form matrices $\Omega=(\Omega_{ij})$, $\Psi=(\Psi_{ij})$. If $d$ denotes the usual exterior derivative, then we have
$$
d\Omega=(d\Omega_{ij}),
$$
and we can define the product of two one-form matrices in a natural way by putting
$$
\Omega\wedge\Psi=\Big(\sum_{k=1}^2\Omega_{ik}\wedge\Psi_{kj}\Big).
$$

All that said, let us consider the following one-form matrix
\bb\label{1.0.7}
\Omega=
\begin{pmatrix}
-i\zeta dx+Adt& qdx+Bdt\\
\\
rdx+Cdt& i\zeta dx-Adt
\end{pmatrix}
\ee
and consider
\bb\label{1.0.8}
dv=\Omega v,
\ee
where $v=(v_1\quad v_2)^t$ (the superscript $t$ means the usual matrix transposition). From \eqref{1.0.7}--\eqref{1.0.8} we obtain
$$
0=d(dv)=d(\Omega v)=(d\Omega-\Omega\wedge\Omega)v,
$$
meaning that
\bb\label{1.0.9}
d\Omega=\Omega\wedge\Omega.
\ee

If we now define two other matrices $X$ and $T$, such that
$$
\Omega=Xdx+Tdt,
$$
then \eqref{1.0.9} gives the equation
\bb\label{1.0.10}
\partial_tX+\partial_xT+[X,T]=0.
\ee

Equation \eqref{1.0.10} is better known as the {\it zero curvature representation}, and it reveals to us that \eqref{1.0.1}--\eqref{1.0.2} can be seen as a geometric object.

Equations \eqref{1.0.8} is a different, but equivalent, way to write \eqref{1.0.1}--\eqref{1.0.2}, whereas \eqref{1.0.9} implies to \eqref{1.0.3}.

The remaining part of the paper is as follows: in Section 2, we revisit the main aspects of the works by Sasaki \cite{sasaki} and Chern and Tenenblat \cite{chern}, connecting equations obtained from the AKNS method and the geometry of abstract 2-manifolds with constant and negative Gaussian curvature. In Section 3, we discuss some recent results from the author \cite{nilay,freire-tito-sam,freire-ch,freire-dp,nazime} that connect the theory developed by Chern and Tenenblat with problems concerning not only equations, but equations subject to some condition (initial data). In some parts, the presentation may closely follow the original references.

\section{Equations describing pseudospherical surfaces}

The study of relations between solutions of equations and the curvature of surfaces dates back to 1862, when, according to \cite[Introduction]{rogers}, Edmond Bour discovered \eqref{1.0.4}. Soon after, Bonnet and Enneper independently rediscovered the same equation in 1867 and 1868, respectively.

Although the works by Bour, Bonnet, and Enneper only considered the sine-Gordon equation, they can be seen as the light dawn of what today is better known as equations describing pseudospherical surfaces or PSS equations (henceforth abbreviated as PSS, also used in the plural).

After the mentioned works, Sasaki's observation, connecting the AKNS system \cite{akns} with the structure equations for a pseudospherical surface, and later the fundamental work by Chern and Tenenblat, where the concept of PSS equations \cite{chern} was established, made connections between certain partial differential equations and smooth two-dimensional Riemannian manifolds of constant Gaussian curvature. Their contribution gained particular significance following a notable insight by Sasaki, who had earlier demonstrated that solutions to specific integrable equations—those arising from the AKNS formalism \cite{akns}—are intimately related to the intrinsic geometry of surfaces with constant negative Gaussian curvature ${\cal K} = -1$ \cite{sasaki}.

By definition, a PSS is a two-dimensional Riemannian manifold whose Gaussian curvature remains constant and negative. According to Hilbert’s theorem \cite[Section 5-11]{carmo}, such manifolds cannot be isometrically immersed in $\mathbb{R}^3$ as complete surfaces, due to their negative curvature—see also \cite[page 439]{neil} and \cite{milnor}. For this reason, they are frequently referred to as {\it abstract} surfaces in the literature.

Once these notions are introduced, let us precisely give meaning to them {\it à la Chern and Tenenblat} \cite{chern}. Let $(x,t)$ be independent variables. A differential equation for a real-valued function $u=u(x,t)$ of order $n$ is generically denoted by 
\bb\label{2.0.1}
{\cal E}(x,t,u,u_{(1)},\cdots,u_{(n)})=0.
\ee

\begin{definition}\label{def2.1}
A differential equation \eqref{2.0.1} is said to describe pseudospherical surfaces, or it is said to be of pseudospherical type, if it is a necessary and sufficient condition for the existence of differentiable functions $f_{ij}$, $1\leq i,j\leq 3$, such that the forms 
\bb\label{2.0.2}
\omega_i=f_{i1}dx+f_{i2}dt,\quad 1\leq i\leq 3,
\ee
satisfy the structure equations of a pseudospherical surface 
\bb\label{2.0.3}
d\omega_1=\omega_3\wedge\omega_2,\quad d\omega_2=\omega_1\wedge\omega_3,\quad d\omega_3=\omega_1\wedge\omega_2.
\ee
\end{definition}

Sasaki noticed in his work \cite{sasaki} that by defining one-forms $\omega_1$, $\omega_2$ and $\omega_3$ suitably from \eqref{1.0.1} and \eqref{1.0.2}, one could obtain the one-form matrix
\bb\label{2.0.4}
\Omega=\frac{1}{2}\begin{pmatrix}
\omega_2 & \omega_1-\omega_3 \\
\omega_1+\omega_3 & -\omega_2
\end{pmatrix},
\ee
and we can define
\bb\label{2.0.5}
\Sigma:=d\Omega-\Omega\wedge\Omega.
\ee

Then $\Sigma=0$ if and only if \eqref{2.0.3} holds. As a result, if \eqref{2.0.1} is the compatibility condition of \eqref{1.0.3}, then its solutions satisfy \eqref{2.0.3}, that is, the structure equations for a two-dimensional Riemannian manifold with first fundamental form, or metric, given by
\bb
I=\omega_1^2+\omega_2^2.
\ee

These foundational results laid the groundwork for a new intersection between surface theory and the analysis of PDEs. Since then, this perspective has shaped a substantial body of literature, including \cite{beals,keti2015,tarcisio,cat,nilay,freire-tito-sam,reyes2000,reyes2002,reyes2006-sel,reyes2006-jde,reyes2011,nazime}, among many others. A summary of Tenenblat's works can be found in \cite[Chapter 1]{keti-book}, whereas \cite{reyes2000} deals with geometric aspects of integrability. A nice survey of PSS equations can be found in the paper by Reyes \cite{reyes2011}.

At least in its roots, the works after \cite{chern} dealing with PSS equations were mostly concerned with integrable equations in the AKNS sense. As such, it was quite natural the presence of a fundamental parameter (in the sense it could not be removed after a Gauge transformation) in the one-forms \eqref{2.0.2}.

From the point of view of AKNS or Lax integrability (see \cite{drazin} for further details), the existence of such a parameter, known as the {\it spectral parameter} due to the fact that it appears in an eigenvalue problem, is somewhat related to the existence of infinitely many linearly independent integrals of motion or conserved quantities. This fact was first shown for the KdV by Miura, Gardner and Kruskal \cite{miura} and then explained in terms of a pair of operators by Lax \cite{lax} and named after him.

From a geometric perspective, the existence of the spectral parameter in the one-forms determined by the solutions of PSS equations implies the existence of a one-parameter family of PSS. This fact leads us to the notion of {\it geometrically integrable equations}, that are those equations whose solutions describe non-trivial one-parameter families of surfaces. To the best of my knowledge, this characterization was first stated by Reyes, see \cite[Definition 2]{reyes2000}, although it dates back to the original work by Chern and Tenenblat, see the comment after Definition 2 in \cite{reyes2000}.

 Despite its importance in the field of integrable systems, soon after the works \cite{sasaki,chern}, the spectral parameter lost importance and the investigation of PSS equations shifted to problems of diverse nature over time, evolving to classification problems and more geometric questions, such as the problems of how immersion of surfaces may depend on jets of finite order, e.g, see \cite{kah-book, kah-cag, kah}.

In practical terms, given a triad of one-forms $\omega_1$, $\omega_2$ and $\omega_3$, and an equation \eqref{2.0.1}, we can check if they describe a PSS surface in the following way: let us define a matrix of one-forms $\Omega$ by \eqref{2.0.4}. If, when restricted to the manifold determined by the  solutions\footnote{The one-forms involved are, in fact, pullbacks of the one-forms by holonomic sections of jets. For a more geometrically rigorous formulation of these concepts, see \cite[Section II]{reyes2000jmp}, see also Remark 3.1.} of \eqref{2.0.1}, the matrix $\Sigma$ vanishes, we say that \eqref{2.0.1} is a PSS equation, and the triad $\{\omega_1,\, \omega_2,\, \omega_3\}$ satisfies the structure equations \eqref{2.0.3} of a PSS with Gaussian curvature ${\cal K}=-1$. As a result, any open, simply connected set $U$ where the one-forms $\omega_1$ and $\omega_2$ are defined and $\omega_1\wedge\omega_2\neq0$ everywhere, is endowed with a PSS structure. We can naively see this condition as a requirement for $\{\omega_1,\omega_2\}$ to be a linearly independent set and thus be the basis of a plane.

\begin{definition}\label{def2.2}
   Assume that \eqref{2.0.1} is an equation describing a PSS and let $u$ be a solution of it. We say that $u$ is a generic solution if the one-forms \eqref{2.0.2} satisfy the condition $\omega_1\wedge\omega_2\neq0$.
\end{definition}

\begin{example}
Let us consider the following one-forms
\bb\label{2.0.7}
\left\{
\ba{lcl}
\omega_1&=&{\frac{1}{\eta}\sin{u}dt,}\\
\\
\omega_2&=&{\eta dx+\frac{1}{\eta}\cos{u}dt},\\
\\
\omega_3&=& u_xdx.
\ea
\right.
\ee

Then we have
$$
\left\{
\ba{lcl}
d\omega_1-\omega_3\wedge\omega_2&\equiv&{0}\\
\\
d\omega_2-\omega_1\wedge\omega_3&\equiv&{0},\\
\\
d\omega_3-\omega_1\wedge\omega_2&=&-(u_{tx}-\sin{u})dx\wedge dt.
\ea
\right.
$$
Thus, on the solutions of the sG equation \eqref{1.0.4}, the forms \eqref{2.0.7} satisfy \eqref{2.0.3}. Since 
$$
\omega_1\wedge\omega_2=-\sin{u}\,dx\wedge dt,
$$
then $u$ will be a generic solution for \eqref{1.0.4} provided that $u\notin\{k\pi,\,\,\,k\in\mathbb{Z}\}=:P$. In particular, if $u$ is a non-constant solution, then it is non-generic, and its domain is ensured with a PSS structure in a simply connected neighborhood of $(x,t)$ whenever $u(x,t)\notin P$.

The generic solutions provide the first fundamental form
$$
I=\eta^2dx^2+2\cos{u}dxdt+\frac{1}{\eta^2}dt^2.
$$
\end{example}

\begin{example}\label{ex2.2}
     Let $m=u-u_{xx}$, $\lambda\in\mathbb{R}\setminus\{0\}$ and consider the one-forms
\bb\label{2.0.8}
\ba{lcl}
\omega_1&=&{\Big(\frac{\lambda}{2}+\frac{1}{2\lambda}-m\Big)dx+\Big(um+\frac{\lambda}{2}u-\frac{u}{2\lambda}-\frac{1}{2}-\frac{\lambda^2}{2}\Big)dt},\\
\\
\omega_2&=&-u_xdt,\\
\\
\omega_3&=&{\Big(m+\frac{1}{2\lambda}-\frac{\lambda}{2}\Big)dx+\Big(\frac{\lambda^2}{2}-\frac{1}{2}-\frac{u}{2\lambda}-\frac{\lambda}{2}u-um\Big)dt}.
\ea
\ee
A straightforward calculation shows that
\bb\label{2.0.9}
\ba{lcl}
d\omega_1-\omega_3\wedge\omega_2&=&\Big(u_t-u_{txx}+3uu_x-2u_xu_{xx}-uu_{xxx}\Big)dx\wedge dt,\\
\\
d\omega_2-\omega_1\wedge\omega_3&=&0,\\
\\
d\omega_3-\omega_1\wedge\omega_2&=&-\Big(u_t-u_{txx}+3uu_x-2u_xu_{xx}-uu_{xxx}\Big)dx\wedge dt,
\ea
\ee
meaning that on the generic solutions of the CH equation 
   \bb\label{1.0.5}
    u_t-u_{txx}+3uu_x=2u_xu_{xx}+uu_{xxx},
    \ee
    or its equivalent form
    \bb\label{1.0.6}
    m_t+2um_x+u_xm=0,
    \ee
the one-forms \eqref{2.0.8} define a two-dimensional Riemannian manifold with first fundamental form
$$
\ba{lcl}
I&=&{\Big[m^2-\Big(\lambda+\frac{1}{\lambda}\Big)m+\frac{1}{4}\Big(\lambda+\frac{1}{\lambda}\Big)^2\Big]dx^2}+\\
\\
&&{\Big[-um^2+\frac{1}{\lambda}um+\frac{1}{4}\Big(\lambda^2-\frac{1}{\lambda^2}\Big)u+\lambda\Big(\frac{\lambda}{2}+\frac{1}{2\lambda}\Big)m-\lambda\Big(\frac{\lambda}{2}+\frac{1}{2\lambda}\Big)^2\Big]dxdt}+\\
\\
&&{\Big[um^2+u^2m\Big(\lambda-\frac{1}{\lambda}\Big)+u_x^2-\lambda\Big(\lambda+\frac{1}{\lambda}\Big)um+\frac{1}{4}\Big(\lambda+\frac{1}{\lambda}\Big)^2u^2-\frac{\lambda}{2}\Big(\lambda+\frac{1}{\lambda}\Big)^2}\\
\\
&&{+\frac{\lambda^2}{4}\Big(\lambda+\frac{1}{\lambda}\Big)^2\Big]dt^2}.
\ea
$$

Let us investigate the generic solutions of the CH equation. It is straightforward to see that
   $$
    \omega_1\wedge\omega_2=-\Big(\frac{\lambda}{2}+\frac{1}{2\lambda}-m\Big)u_x dx\wedge dt.
   $$
    Provided that $u$ is a solution of the CH equation, then $\omega_1\wedge\omega_2=0$ somewhere if and only if 
    \bb\label{2.0.10}
    m=\frac{\lambda}{2}+\frac{1}{2\lambda}
    \ee
    or
    \bb\label{2.0.11}
    u_x=0.
    \ee
\end{example}

We close the present section with some remarks.

\begin{remark}[Non-uniqueness of the one-forms]
It is possible that a PSS equation be the compatibility condition for more than one triad of one forms satisfying the structure equations \eqref{2.0.3}. To illustrate this observation, consider the one-forms
$$
\ba{lcl}
\omega_1&=&{\Big(u_{xx}-u-\beta+\frac{\beta}{\eta^2}-\frac{1}{\eta^2}\Big)dx}\\
\\
&&+{\Big(-\frac{\beta}{\eta}u_x-\frac{\beta}{\eta^2}+u^2-1+\beta u+\frac{u_x}{\eta}+\frac{1}{\eta^2}-uu_{xx}\Big)dt}\\
\\
\omega_2&=&{\eta dx+\Big(-\frac{\beta}{\eta}-\eta u+\frac{1}{\eta}+u_x\Big)dt},\\
\\
\omega_3&=&{\Big(u_{xx}-u+1\Big)dx+\Big(\frac{\beta}{\eta^2}u+u^2=uu_{xx}+\frac{1}{\eta^2}+\frac{u_x}{\eta}-\frac{u}{\eta^2}-u-\frac{\beta}{\eta^2}-\frac{\beta}{\eta}u_x\Big)dt},
\ea
$$
where $\eta^4-\eta^2+\beta^2\eta^2=\beta^2+1-2\beta$. In \cite[Theorem 1]{reyes2002} it was shown that if we substitute the one-forms above into \eqref{2.0.3}, we then obtain the CH equation \eqref{1.0.5} as the compatibility condition. Compare with the representation in  Example \ref{ex2.2}.
\end{remark}

\begin{remark}[{Geometric integrability and AKNS integrability are not equivalent}]
The theory of PSS equations, even though purely geometric, was born in the context of integrable systems. That is why the earlier papers in the area were concerned with equations that are integrable. The intense research carried out by K. Tenenblat and her co-workers shows that the land of PSS equations is larger than that of integrable equations in the AKNS sense. 

It is also worth mentioning that even though geometric integrability and AKNS integrability are related and shared by a number of equations, they are not equivalent. The Cavalcanti and Tenenblat equation
$$
u_t=(u_x^{-1/2})_{xx}+u_x^{3/2},
$$
discovered by Cavalcanti and Tenenblat \cite{cat}, is the compatibility condition for the one-forms
$$
\ba{lcl}
\omega_1&=&\eta\sinh{u}\,dx+\eta\Big((u_x ^{-1/2})_x\cosh{u}+(u_x^{1/2}-\eta u_x^{-1/2})\sinh{u}\Big)dt,\\
\\
\omega_2&=&\eta dx-\eta^2u_x^{-1/2}dt,\\
\\
\omega_3&=&\eta\cosh{u}\,dx+\eta\Big((u_x ^{-1/2})_x\sinh{u}+(u_x^{1/2}-\eta u_x^{-1/2})\cosh{u}\Big)dt,
\ea
$$
satisfies the structure equations \eqref{2.0.3}, but it is not of the AKNS type \cite[page 79]{reyes2000}.
\end{remark}

\begin{remark}[{An equation describing pseudospherical surfaces is not necessarily geometrically integrable}]
 If an equation is geometrically integrable, then it describes one-parameter families of PSS. This is due to the presence of a (at least one) parameter in the one-forms that cannot be removed by a gauge transformation $X\mapsto SXS^{-1}$, $ T\mapsto STS^{-1}$, where $X$ and $T$ satisfy the ZCR \eqref{1.0.10}. We observe that such a gauge transformation would change the corresponding one-forms but does not affect geometric integrability.

Let us consider the one-forms (see \cite[Example 2.8]{tarcisio})
\bb\label{2.12}
\ba{lcl}
\omega_1&=&(u-u_{xx})dx+(u_x^2-2uu_x+uu_{xx})dt,\\
\\
\omega_2&=&\Big(\mu(u-u_{xx})\pm 2\sqrt{1+\mu^2}\Big)dx+\mu(u_x^2-2uu_x+uu_{xx})dt,\\
\\
\omega_3&=&\Big(\pm\sqrt{1+\mu^2}(u-u_{xx})+2\mu\Big)dx\pm\sqrt{1+\mu^2}(u_x^2-2uu_x+uu_{xx})dt.
\ea
\ee
A straightforward calculation shows that
$$
\ba{lcl}
d\omega_1-\omega_3\wedge\omega_2&=&{\Big(u_t-u_{txx}+4uu_x-3u_xu_{xx}-uu_{xxx}\Big)dx\wedge dt,}\\
\\
d\omega_2-\omega_1\wedge\omega_3&=&0,\\
\\
d\omega_3-\omega_1\wedge\omega_2&=&{-\Big(u_t-u_{txx}+4uu_x-3u_xu_{xx}-uu_{xxx}\Big)dx\wedge dt}.
\ea
$$

Therefore, the one-forms \eqref{2.12} satisfy the structure equations for a PSS provided that $u$ is a solution of the Degasperis-Procesi equation
$$
u_t-u_{txx}+4uu_x=3u_xu_{xx}+uu_{xxx}.
$$

The presence of the parameter $\mu$ in \eqref{2.12} might suggest that the Degasperis-Procesi equation is geometrically integrable. However, this is not the case, at least with the forms \eqref{2.12}.

Under a gauge transformation, see \cite[Section 8]{freire-dp}, the one-forms \eqref{2.12} are equivalent to
$$
\ba{lcl}
    \theta_1&=&-2dx,\\
    \\
    \theta_2&=&{\Big(1+\frac{u-u_{xx}}{2}\Big)dx+\frac{u_x^2-2uu_x+uu_{xx}}{2}dt},\\
    \\
    \theta_3&=&{\Big(1+\frac{u-u_{xx}}{2}\Big)dx+\frac{u_x^2-2uu_x+uu_{xx}}{2}dt,}
    \ea
$$
which does not depend on any parameter. It is important to mention that this fact does not {\it imply} that the Degasperis-Procesi equation is not geometrically integrable! It does imply that it is not geometrically integrable, considering \eqref{2.12}, because the parameter in those forms can be removed. A discussion about geometric integrability of the Degasperis-Procesi equation can be found in \cite[Discussion]{freire-dp}.
\end{remark}

\begin{remark}[{PSS equations are a proper category}]
 The set of PSS equations is larger than the set of integrable equations {\it à la} AKNS integrability. In fact, those PSS equations that are integrable (in any sense we may consider) represent a small, but quite important, number of what today we know to be PSS equations. This distinction makes evident that the existence of equations being the compatibility condition for the structural equations \eqref{2.0.3} reveals a universal geometric feature, whereas the integrability aspect, despite its relevance, is a specific subcase. 
 \end{remark}

\begin{remark}[{PSS from smooth solutions}]
 None of the definitions \ref{def2.1} or \ref{def2.2} made reference of the regularity of the objects involved. This is intentional: they follow the spirit of papers in geometry, where usually objects are assumed to be $C^\infty$. In particular, all solutions involved are $C^\infty$ and so are the forms \eqref{2.0.2}.
\end{remark}

\begin{remark}[{PSS equations and Cauchy problems}]
The developments in the field of PSS equations had been restricted to equations or systems of equations. Additional information, such as boundary or initial conditions, was not considered until a few years ago. This point will be better explored in the next section.
\end{remark}

\begin{remark}[{Limitations of the theory of PSS equations}]
The theory of PSS equations has traditionally been formulated under the assumption that the solutions defining the first fundamental form possess $C^\infty$ regularity. Although this smoothness requirement is not always explicitly stated, it can be traced back to the foundational work of Chern and Tenenblat \cite[page 55, second paragraph]{chern}. In contrast, the discussion in the preceding remark highlights a tension between this framework and the class of solutions considered in \cite{const1998-1,const1998-2,const2000-1}, where the smoothness assumption is incompatible with the finite regularity typically associated with Cauchy problems. Consequently, surfaces potentially generated by solutions exhibiting wave breaking remained unexplored within this geometric context until \cite{freire-ch}. 
\end{remark}

\section{Finite regularity of the metric}

In this section, we address problems concerning PSS equations and some developments made by the author in \cite{nilay,freire-tito-sam,freire-ch,freire-dp,freire-arxiv,nazime}. The CH equation was the main motivation for such improvements, and thus, parts of the text are now more focused on its literature. For a review on some aspects of the CH equation, see \cite{freire-cm}.

A smooth solution of a PSS equation yields smooth one-forms $\omega_1, \omega_2, \omega_3$, and consequently, the associated first fundamental form inherits the same degree of regularity. In contrast, the solutions considered by Constantin \cite{const2000-1} are not necessarily $C^\infty$, which marks a significant distinction from works such as \cite{tarcisio,keti2015,reyes2000,reyes2002,reyes2006-sel,reyes2006-jde,reyes2011,sasaki}, where smoothness is a standing assumption.

In the literature on PDEs and PSS structures, the uniqueness of solutions is not often addressed. Consequently, whether a specific first fundamental form may correspond to one or multiple solutions of PSS equations remains, overall, an open topic.

This distinction motivates a deeper investigation into the relationship between certain regular curves and the graphs of solutions of the CH equation. More precisely, one may ask under which conditions a regular curve $\gamma(x)=(x,0,u_0(x))$ in space can uniquely determine such a graph, together with a compatible set of one-forms. In certain cases, uniqueness can indeed be established. This situation essentially corresponds to solving the CH equation with prescribed initial data $u_0$, which naturally determines a curve of the form $(x,0,u_0(x))$. However, the requirement of uniqueness imposes constraints on the class of curves of this type that can be considered.

As previously noted, solutions to Cauchy problems for the CH equation do not always exhibit smoothness \cite{const1998-1,const1998-2,const2000-1}, that is, their solutions may not be $C^\infty$. Some issues naturally arise from this observation:
\begin{itemize}
    \item are definitions \ref{def2.2} and \ref{def2.1} still valid in the absence of $C^\infty$ regularity?
    \item if the answer is positive, what is the minimal regularity necessary to ensure the well-definition of the associated one-forms?
\end{itemize}

In light of the preceding comments, from a geometric perspective, PSS equations are not usually dealt with other conditions or restrictions, such as initial conditions. Exceptions are rare, and to the best of the author's knowledge, \cite{nazime} stands out as a first example of the Cauchy problem being considered in conjunction with the geometry of PSS equations. There, PSS surfaces derived from solutions to a specific Cauchy problem were considered, but the solutions in question were analytic ($C^\omega$), and hence smooth. Nevertheless, \cite{nazime} also demonstrated the existence of a strip in which the one-forms are defined and the condition $\omega_1 \wedge \omega_2 \neq 0$ is satisfied.

It is worth recalling that, for a solution to define a first fundamental form of a surface, the associated one-forms must be linearly independent over simply connected domains in the plane \cite[Theorem 4.39]{cle}. This requirement prompts further examination of how the geometry of the solution's graph evolves, particularly with respect to the domain where a PSS structure can be meaningfully defined.

Despite the considerations in \cite{nazime}, the solutions considered therein were still $C^\infty$, so that the framework proposed by Tenenblat and co-workers over the years was still valid. The situation, however, is significantly different when solutions of the CH equation are taken into consideration.

Cauchy problems involving the CH equation are usually established with finite regularity, since they are $C^1$ in time \cite{const1998-1,const1998-2,const2000-1, CE}. Moreover, solutions of the CH equation may develop singularities in finite time. For example, Constantin and Escher \cite[Example 4.3]{const1998-2} demonstrated that even smooth initial data may lead to wave breaking in finite time. The formation of such singularities is typically expressed through the blow-up of the slope of the solution. The presence of finite regularity certainly brings some difficulties to Definition \ref{def2.1}, but with Definition \ref{def2.2} the issue becomes dramatic, since it is not clear if we still have an intrinsic geometry defined by \eqref{2.0.2}--\eqref{2.0.3} whenever the solution of the equation has finite regularity. 

The initial data (which geometrically gives a curve) provides the regularity of the solution because it determines the space function where the solutions belong to. This is precisely the source of the problems mentioned above: the space where a given solution is considered. From a geometric perspective, provided that the one-forms \eqref{2.0.2} are $C^1$ (see \cite[Theorem 10-19, page 232]{gug} and also \cite[Theorem 10-18, page 232]{gug}), we \textit{do} have an abstract surface defined. Therefore, as long as we have a solution belonging to a function space respecting this constraint, the theory developed so far continues to work.

With all of this in mind, we have a clue to modify Definition \ref{def2.1} and make it compatible when we are out of $C^\infty$ structures: the definition should take into account {\it where} the solution is. This is, in fact, done in its original assumption, since it implicitly assumes $C^\infty$ solutions. If we allow other solutions, we can then bring modifications in the notion of PSS equations and generic solutions, e.g., see \cite{freire-ch,freire-dp}.

\begin{definition}[{$C^k$ PSS modelled by ${\cal B}$ and ${\cal B}-$PSS equation, see \cite[Definition 2.1]{freire-ch}}]\label{def3.1}
Let ${\cal B}\subseteq C^k$ be a function space. A differential equation \eqref{2.0.1} for a dependent variable (function) $u\in{\cal B}$ is said to describe a pseudospherical surface of Gaussian curvature ${\cal K}=-1$ and class $C^k$ modelled by ${\cal B}$, $k\in\mathbb{N}$, or it is said to be of ${\cal B}-$pseudospherical type (${\cal B}$-PSS equation, for short), if it is a necessary and sufficient condition for the existence of functions $f_{ij}=f_{ij}(x,t,u,u_{(1)},\cdots,u_{(n)})$, $1\leq i\leq 3,\,\,1\leq j\leq 2$, depending on $u$ and its derivatives up to a finite order $n$, such that:
\begin{itemize}
    \item the functions $f_{ij}$ are $C^k$ with respect to their arguments;
    \item the forms \eqref{2.0.2} satisfy the structure equations \eqref{2.0.3} of a pseudospherical surface of Gaussian curvature ${\cal K}=-1$;
    \item $\omega_1\wedge\omega_2\neq0$.
\end{itemize}
\end{definition}

If the function space is the space of $C^\infty$ functions, or it is not exactly or clearly defined, and no confusion is possible, we then say PSS equation instead of ${\cal B}-$PSS equation.

\begin{remark}
The one-forms involved in Definition \ref{def3.1} are pullbacks of differential forms on jet spaces; see \cite[Section II]{reyes2000jmp}.
\end{remark}

\begin{remark}
One might wonder whether the order of the derivatives appearing in the functions $f_{ij}$ in Definition \ref{def3.1} is related to the order of the equation. Most commonly, they are of a smaller order. However, this is not a rule. For example, let us consider the one-forms
$$
\ba{lcl}
\omega_1&=&-2dx,\\
\\
\omega_2=\omega_3&=&{\Big(1-2u+u_x+2u_{xx}-u_{xxx}\Big)dx}\\
\\
&+&{\Big(16u_x^2-16uu_x+16uu_{xx}-16u_xu_{xx}-4uu_{xxx}+2u_{xx}^2+2u_xu_{xxx}\Big)dt},
\ea
$$
which were reported in \cite[Theorem 2.1]{raspa-jde}.

Their coefficients depend on up to third-order derivatives, and a lengthy calculation shows that
$$
\ba{lcl}
  d\omega_1&-&\omega_3\wedge\omega_2\equiv0,\\
    d\omega_2&-&\omega_1\wedge\omega_3=d\omega_3-\omega_1\wedge\omega_2\\
    &=&(2-\partial_x)\Big(u_t-u_{txx}-\big(16uu_x-8u_xu_{xx}+2u_{xx}^2-4uu_{xxx}+2u_xu_{xxx}\big)\Big)dx\wedge dt.
\ea
$$
Therefore, on the solutions of the equation
$$u_t-u_{txx}=16uu_x-8u_xu_{xx}+2u_{xx}^2-4uu_{xxx}+2u_xu_{xxx}$$
the one-forms satisfy the structure equations for a PSS.
\end{remark}

From now on, we closely follow the ideas presented in \cite{freire-ch}. Our presentation is more concerned with the geometric nature of the results and we avoid technicalities with functional analysis. However, we guide the reader for \cite{freire-ch} to clarify certain aspects on the function spaces involved, such as the Sobolev spaces $H^s(\mathbb{R})$ and other related tools.

\begin{definition}[{Generic solution, \cite[Definition 2.2]{freire-ch}}]\label{def3.2}
A function $u:U\rightarrow\mathbb{R}$ is called a \textit{generic solution} for the ${\cal B}-$PSS equation \eqref{2.0.1} provided that:
\begin{enumerate}
\item $u\in{\cal B}$;
\item it is a solution of the equation;
\item the one-forms \eqref{2.0.2} are $C^k$ on $U$;
\item there exists at least one simply connected open set $\Omega\subseteq U$ such that $\omega_1\wedge\omega_2\big|_{p}\neq0$, for each $p\in\Omega$.
\end{enumerate}
Otherwise, $u$ is said to be \textit{non-generic}.
\end{definition}

\begin{remark}
The condition $\omega_1\wedge\omega_2\big|_{p}\neq0$ in Definition \ref{def3.2} means that
$$\Big(f_{11}f_{22}-f_{12}f_{21}\Big)(p,u(p),u_{(1)}(p),\cdots,u_{(n)}(p))\neq0.$$
\end{remark}

Henceforth, we consider the results in Example \ref{ex2.2}

\begin{lemma}\label{lema3.1}
    Assume that $u\in C^{0}(H^{4}(\mathbb{R}),[0,T))\cap C^{1}(H^{3}(\mathbb{R}),[0,T))$ is a non-trivial solution of the CH equation \eqref{1.0.6}. If $u_x\neq 0$ on some open set $\Omega$, then we cannot have $m=c\neq 0$, where $c$ is a constant.
\end{lemma}

\begin{proof}
    Suppose $m=c\neq 0$ is a constant on some open and non-empty set $\Omega$, then necessarily $m\neq 0$ and \eqref{1.0.6} implies \eqref{2.0.11}. On the other hand, if $u$ is a solution of \eqref{1.0.6} such that $u_x=0$ on an open set $\Omega$, then $u_x=u_{xx}=u_{xxx}=u_{txx}=0$ and $u_t=0$ on $\Omega$, which, once substituted into \eqref{1.0.6}, implies $u(x,t)=c$.
\end{proof}

\begin{theorem}[{\cite[Theorem 2.2]{freire-ch}}]\label{teo3.1}
    Let $u_0\in H^4(\mathbb{R})$ be a non-trivial initial datum and $u$ be the corresponding solution of the CH equation subject to $u(x,0)=u_0(x)$. Then there exists a strip ${\cal S}=\mathbb{R}\times(0,T)$, uniquely defined by $u_0$, such that:
    \begin{enumerate}
        \item the one-forms \eqref{2.0.8} are defined on ${\cal S}$;
        \item there exist at least two open and disjoint discs endowed with the structure of a $C^1$ PSS modelled by $C^{0}(H^{4}(\mathbb{R}),[0,T))\cap C^{1}(H^{3}(\mathbb{R}),[0,T))$.
    \end{enumerate}
\end{theorem}

\begin{proof}
    Let $u_0\in H^4(\mathbb{R})$. From \cite[Proposition 2.7]{const1998-1} with $s=4$ we can find a number $T>0$, uniquely determined by $u_0$, and determine a unique function \[u\in C^{0}(H^{4}(\mathbb{R}),[0,T))\cap C^{1}(H^{3}(\mathbb{R}),[0,T))\] that is a solution of \eqref{1.0.5}. The lifespan $T$ determined by the initial datum uniquely defines the strip mentioned in the statement of the theorem, whereas the solution $u$, granted by \cite[Proposition 2.7]{const1998-1}, ensures that \eqref{2.0.9} are $C^1$ one-forms defined on ${\cal S}$.

    Since $u\in C^{0}(H^{4}(\mathbb{R}),[0,T))\cap C^{1}(H^{3}(\mathbb{R}),[0,T))$, we have $u_x\in H^3(\mathbb{R})$ for each fixed $t$ in $[0,T)$. Moreover, since the $H^1(\mathbb{R})$-norm of the solutions of the CH equation is invariant, and the fact that $u_0$ is non-trivial, implies that for any value of $t$ for which the solution exists, we have
    \begin{equation}\label{3.0.1}
    \|u(\cdot,t)\|_{H^1(\mathbb{R})}=\|u_0(\cdot)\|_{H^1(\mathbb{R})}>0,
    \end{equation}
    telling us that
    \begin{equation}\label{3.0.2}
    \lim_{|x|\rightarrow\infty}u(x,t)=0\quad\text{and}\quad\lim_{|x|\rightarrow\infty}u_x(x,t)=0.
    \end{equation}
    As a result, for each fixed $t\in[0,T)$, the function $x\mapsto u_x(x,t)$ has a minimum and a maximum. Let $a_t$ and $b_t$, with $a_t<b_t$, be points where these extrema are achieved (note that they are not necessarily unique). Let $I(t):=u_x(a_t,t)$, $S(t):=u_x(b_t,t)$ and $h(t):=I(t)S(t)$.

    \textbf{Claim.} We claim that $h(t)<0$.

    Let $a_t$ and $b_t$ be the distinct points we found earlier. Without loss of generality, we may assume $u_x(a_t,t)<0<u_x(b_t,t)$. By continuity, we can find $\epsilon>0$ such that $B_\epsilon(a_t,t)$ and $B_\epsilon(b_t,t)$ are disjoint, their union is contained in ${\cal S}$, and
    $$u_x\Big|_{B_\epsilon(b_t,t)}>0\quad\text{and}\quad u_x\Big|_{B_\epsilon(a_t,t)}<0.$$

    By Lemma \ref{lema3.1}, $m$ cannot be constant on $B_\epsilon(b_t,t)$. As a result, for some $(y_0,t_0)$ in $B_\epsilon(b_t,t)$ we have $m(y_0,t_0)\neq \frac{\lambda}{2}+\frac{1}{2\lambda}$. Then, by continuity, we can find $\epsilon'\in(0,\epsilon]$ such that $m(x,t)\neq \frac{\lambda}{2}+\frac{1}{2\lambda}$, for all $(x,t)\in B_{\epsilon'}(y_0,t_0)\subseteq B_\epsilon(b_t,t)$. Let $B_1:=B_{\epsilon'}(y_0,t_0)$. The same argument shows the existence of another disc $B_2\subseteq B_\epsilon(a_t,t)$ where $m\neq \frac{\lambda}{2}+\frac{1}{2\lambda}$.

    Therefore, by the first part of the theorem, the one-forms \eqref{2.0.8} are defined on these discs, while the conditions above say that $\omega_1\wedge\omega_2\big|_{B_1}\neq0$, $\omega_1\wedge\omega_2\big|_{B_2}\neq0$, and they do not vanish at any point of $B_1\cup B_2$.

    Now, we prove the claim by showing that we cannot have $h(t)\geq0$. If we had $h(t_0)\geq0$ for some $t_0\in[0,T)$, then either $u_x(x,t_0)\geq0$ or $u_x(x,t_0)\leq0$ for all $x\in\mathbb{R}$, meaning that $u$ is either non-decreasing or non-increasing. Therefore, \eqref{3.0.2} implies $u(\cdot,t_0)\equiv0$, which contradicts \eqref{3.0.1}.
\end{proof}

We have just shown that any non-trivial initial datum defines uniquely a strip ${\cal S}$ containing open sets that can be endowed with the structure of a PSS. Our task now is to look for singularities of the metric defined by the forms \eqref{2.0.8}. By a singularity, we mean a point $p$ for which the metric is singular, in the sense that it is no longer a positive definite bilinear form. This situation occurs whenever the one-forms $\omega_1$ and $\omega_2$ are not linearly independent at $p$, that is, $\omega_1\wedge\omega_2\big|_p=0$.

\begin{theorem}[{\cite[Theorem 2.3]{freire-ch}}]\label{teo3.2} 
    Assume that $u_0\in H^4(\mathbb{R})$ is a non-trivial initial datum, $u$ the corresponding solution of \eqref{1.0.5}, and let ${\cal S}$ be the strip determined in Theorem \ref{teo3.1}.

    \begin{enumerate}
        \item If there exists some open set $\Omega\subseteq{\cal S}$ such that $\omega_1\wedge\omega_2\big|_\Omega=0$, then necessarily $u(x,t)=c$ for some $c\in\mathbb{R}$, $(x,t)\in\Omega$, and $\Omega$ is a proper subset of ${\cal S}$;
        \item For each $t\in(0,T)$, there exists at least one point $c_t\in\mathbb{R}$ such that $\omega_1\wedge\omega_2\big|_{(c_t,t)}=0$.
    \end{enumerate}
\end{theorem}

\begin{proof}
    From Example \ref{ex2.2}, $u$ is a non-generic solution of the CH equation provided that $u_x=0$ or $m=const.$ For convenience, let us consider the CH equation in the form \eqref{1.0.6}. If we assume that \eqref{2.0.10} holds on some open and non-empty set $\Omega$, then necessarily $m\neq0$ and \eqref{1.0.6} implies \eqref{2.0.11}. On the other hand, Lemma \ref{lema3.1} implies that if $m$ is constant on some open set, then $u(x,t)$ coincides with $m$ within the same open set.

    Let us show the existence of at least one point $c_t$ such that $u_x(c_t,t)=0$ for each $t\in[0,T)$. Let $h$, $a_t$, and $b_t$ be as in the proof of Theorem \ref{teo3.1}. The two points $a_t$ and $b_t$ cannot be the same, and the intermediate value theorem ensures the existence of $c_t\in(a_t,b_t)$ such that $u_x(c_t,t)=0$.

    Finally, let us prove that $\Omega$ must be a proper subset of ${\cal S}$. In fact, if we had $\Omega={\cal S}$, then we would be forced to conclude that $c=0$ from \eqref{3.0.2}, which would then lead to a contradiction with \eqref{3.0.1}.
\end{proof}

The geometric phenomena described in the Theorems 3.1 and 3.2 are not particular aspects related to the CH equation. They also appear, for instance, in the study of the Degasperis-Procesi equation and its connections with PSS surfaces, see \cite{freire-dp}. 

\section{Further remarks}

\begin{remark}
The investigation of PSS equations essentially began with the papers by Sasaki \cite{sasaki} and Chern and Tenenblat \cite{chern}. Both papers were motivated by the geometric nature of certain integrable equations in the sense described in \cite{akns}.  
\end{remark}

\begin{remark}
Chern and Tenenblat \cite{chern}, although motivated by Sasaki \cite{sasaki}, started systematic geometric studies of integrable systems from the point of view of two-dimensional Riemannian manifolds. Consequently, nowadays, we know that PSS equations are not only restricted to AKNS equations. In fact, equations without any integrable properties can be PSS equations, e.g., see \cite{keti2015} for some examples of non-integrable equations that are PSS equations. 
\end{remark}

\begin{remark}
E. Reyes has made fundamental contributions to the study of PSS equations and integrable systems \cite{benson,reyes2000jmp,reyes2000,reyes2002,reyes2006-sel,reyes2006-jde,reyes2011}. In \cite{reyes2006-sel} the connections between PSS equations and integrability are explored, while in \cite{reyes2011} a nice review on PSS in light of the legacy of Chern and Tenenblat contributions is presented. 
\end{remark}
 
\begin{remark}
Reyes proved that the CH equation is geometrically integrable \cite{reyes2002}, which is the starting point for the results established in \cite{freire-ch}, bringing qualitative information to the surface from the properties of the initial data of Cauchy problems involving the CH equation. 
\end{remark}

\begin{remark}
A topic not treated in this paper is the question of immersion, into the Euclidean space, of the abstract surfaces determined by the solutions of PSS equations. It is well known that PSS can be (locally) immersed in the three-dimensional space. In particular, there are works studying the immersion of surfaces determined by equations, see \cite{tarcisio,freire-dp,kah-book, kah-cag, kah,reyes2000jmp} and references therein. 
\end{remark}

\begin{remark}
The results reported in \cite{tarcisio} say that the immersion of a PSS determined by a generic solution of the CH equation must depend on an infinite number of derivatives of the solution. However, the results established in \cite{freire-ch}, see also \cite{nilay,freire-dp,freire-arxiv} tell us that the manifold may have a metric with finite regularity (see Theorem \ref{teo3.1}). This is, however, not a contradiction. In \cite{tarcisio} it is assumed that the solutions are $C^\infty$ and then, the natural {\it locus} for all the involved functions and forms are the jet space, eventually an infinite jet space. The results in \cite{freire-ch} are obtained in a context of finite jets, meaning that a second fundamental form cannot have arbitrary derivatives. In particular, the results in \cite{freire-ch} open doors for us to investigate second fundamental forms having a different dependence on the field variables. 
\end{remark}

\begin{remark}
From a topological viewpoint, all surfaces in Figure \ref{fig1} possess singularities. On the other hand, the results established in \cite{freire-ch,freire-dp} made possible the ``discovery'' of a new kind of singularity: the blow-up of the metric. It is not clear to the authors whether the topological singularity of the surfaces in Figure \ref{fig1} might be somewhat related to the blow up of the metric of the surfaces described by the solutions of the Camassa-Holm or Degasperis-Procesi equations experiencing wave-breaking. This is a point that should be better clarified in the future.
\end{remark}

\section{Discussion}

The development of the theory of equations describing pseudospherical surfaces, as traced in this review, illustrates a profound and dynamic synergy between differential geometry and non-linear analysis. It was brought to light due to the elegant structure of the AKNS system \cite{akns}, which played a vital role in the foundational works of Sasaki \cite{sasaki} and Chern–Tenenblat \cite{chern} that aimed to provide a geometric framework for the (at that time) emerging theory of integrable systems. Over time, however, the scope of the theory has expanded significantly. It is now widely recognized that the property of describing a pseudospherical surface is not confined to integrable equations. As shown by many works developed by Tenenblat and co-workers, e.g., see \cite{tarcisio}, certain differential equations possess this geometric attribute without admitting a Lax pair or soliton solutions, demonstrating that the class of PSS equations is larger than the class of integrable equations.

The contributions of Reyes \cite{reyes2000jmp,reyes2000,reyes2002,reyes2006-sel,reyes2006-jde,reyes2011,benson} have been pivotal in refining the classification of PSS equations and clarifying the relationship between geometry and integrability, including extrinsic aspects \cite{reyes2000jmp,benson}. His introduction of the concept of geometric integrability provided a key distinction: it isolates those PSS equations whose associated pseudospherical structure depends essentially on a non-removable parameter. Such a parameter bridges the general geometric definition with the analytical machinery of the Inverse Scattering Transform, offering a geometric explanation for the presence of infinite hierarchies of conservation laws and Bäcklund transformations, which emerge from isospectral deformations of the surface.

More recently, a gap between the analysis of PDEs and geometry has been addressed by recent contributions of the author. A central issue addressed by recent research, and highlighted in this review, is the question of the regularity of solutions of PSS equations and their impact on the geometry determined by them. Classical formulations \cite{chern,tarcisio,reyes2002} were set assuming $C^\infty$ solutions, meaning that the differential forms are defined on infinite-order jet spaces—a natural assumption for classical soliton equations. However, the study of wave-breaking phenomena in models such as the CH equation requires moving beyond this smooth paradigm. In these models, solutions remain continuous while their derivatives become unbounded in finite time (breaking time), a behavior not captured by the infinite jets setting where the theory originally proposed by Chern and Tenenblat \cite{chern} was built. The results in \cite{freire-ch,freire-dp} show that the pseudospherical structure remains well-defined even for metrics of finite regularity, effectively transitioning the theory from infinite to finite jet spaces. This adaptation is not a contradiction but an essential evolution, enabling the geometric description of singular dynamics characteristic of peakon equations.

It is not yet clear to the author if the modifications introduced in Definition~\ref{def2.1} and Definition~\ref{def2.2} are sufficient for studying more complex and challenging solutions of equations like the CH. An example is the peakon solutions: given that they are distributional solutions of a non-local evolution for of the CH equation, they do not describe any surface in the sense of Definition \ref{def2.1}. On the one hand, this shows an essential limitation of these recently introduced innovations. On the other hand, the requirement of finite jets in the definition of PSS tells us that we can expect issues when the problem of immersions is studied in conjunction with Cauchy problems, at least for CH type equations. This is a topic to be better understood in the field.

\section{Conclusion}

In this survey, we revisited the notions of PSS equations from its roots, based on the theory of integrable systems, passing through the fundamental contributions by Tenenblat and Reyes, until recent works of the author, who have incorporated tools of qualitative theory of PDEs in the study of PSS.


\end{document}